\documentclass[a4paper,12pt]{article}
\usepackage{amsmath,amssymb,amsfonts,amscd,amsthm}
\usepackage{a4wide}
\usepackage{graphicx}
\usepackage{dsfont}
\usepackage[margin=25mm]{geometry}
\usepackage{times}
\usepackage{blindtext}
\usepackage[titletoc, toc, title, page]{appendix}
\usepackage{color}
\usepackage{epsfig} 
\usepackage{enumerate}
\usepackage{fancyhdr}		
\usepackage{pgf,tikz}
\usepackage{sectsty} 
\definecolor{bleu1}{RGB}{0,57,128}
\definecolor{emerald}{rgb}{0.31, 0.78, 0.47}
\definecolor{deepjunglegreen}{rgb}{0.0, 0.29, 0.29}
\usepackage[utf8]{inputenc}
\usepackage[english]{babel}

\usepackage[colorlinks,linkcolor=deepjunglegreen,anchorcolor=green,citecolor=bleu1]{hyperref}

\usepackage{xcolor}
\definecolor{dgreen}{rgb}{0.1,0.6,0.1}

\definecolor{bluegray}{rgb}{0.4, 0.6, 0.8}

\usepackage{xcolor}

\newtheorem{theorem}{Theorem}                             
\newtheorem{lemma}{Lemma}

\newtheorem{corollary}{Corollary}

\newtheorem{remark}{Remark}

\numberwithin{equation}{section}

\makeatletter
\def\author#1{\gdef\autrun{\def\and{\unskip, }#1}\gdef\@author{#1}}

\makeatother

\def\keywords#1{\par\medskip
\noindent\textbf{Keywords.} #1}

\renewenvironment{abstract}
 {\small
  \begin{center}
  \bfseries \abstractname\vspace{-.5em}\vspace{0pt}
  \end{center}
  \list{}{
    \setlength{\leftmargin}{2cm}%
    \setlength{\rightmargin}{\leftmargin}%
  }%
  \item\relax}
 {\endlist}

\renewenvironment{proof}{\noindent {\textit{Proof.}}}{\qed\vskip 0.2cm}

\makeatletter
\newenvironment{proofof}[1]{\par
  \pushQED{\qed}%
  \normalfont \topsep6\p@\@plus6\p@\relax
  \trivlist
  \item[\hskip\labelsep
    \textit{\textit{Proof of} #1\@addpunct{.}}]\ignorespaces
}{%
  \popQED\endtrivlist\@endpefalse
}
\makeatother

\newcommand{\C}{\mathbb C}
\newcommand{\R}{\mathbb R}

\newcommand{\Z}{\mathbb Z}
\newcommand{\N}{\mathbb N}

\newcommand{\T}{\mathbb{T}}

\renewcommand{\inf}{\mathop{\mbox{{inf}}}
	\renewcommand{\lim}{\mathop{\mbox{lim}}}
	\renewcommand{\max}{\mathop{\mbox{ma}x}}}
\renewcommand{\min}{\mathop{\mbox{min}}}

\providecommand{\abs}[1]{\lvert#1\rvert}	
\providecommand{\norm}[1]{\lVert#1\rVert}

\begin{document}

\title{Nearly-optimal effective stability estimates around Diophantine tori of Hölder Hamiltonians}

\author{Santiago Barbieri \and Gerard Farré}
\date{}

\maketitle

 \begin{abstract}
 We prove that the solutions of Hölder-differentiable Hamiltonian systems, associated to initial conditions in a small ball of radius $\rho>0$ around a Lagrangian, $(\gamma,\tau)-$Diophantine, quasi-periodic torus, are stable over a time $t^{\text{stab}}\simeq 1/(|\rho|^{1+\frac{\ell-1}{\tau+1}}|\ln \rho|^{\ell-1})$, where $\ell>2d+1, \ell \in \R$, is the regularity, and $d$ is the number of degrees of freedom. In the finitely differentiable case (for integer $\ell$), this result improves the previously known effective stability bounds around Diophantine tori.  Moreover, by a previous work based on the Anosov-Katok construction, it is known that for any $\varepsilon>0$ there exists a $C^\ell$-Hamiltonian, with $ \ell\ge 3$, admitting a sequence of solutions starting at distance $\rho_n \to 0$ from a $(\gamma,\tau)$-Diophantine torus that diffuse in a time of order $t^{\text{diff}}_n\simeq 1/(|\rho_n|^{1+\frac{\ell-1}{\tau+1}+\varepsilon})$.
  Therefore the stability estimates that we show are optimal up to an arbitrarily small polynomial correction.
\end{abstract}

\section{Introduction, setting, and main result}
In this note, we study  the effective stability of Hölder Hamiltonians around Lagrangian, Diophantine, quasi-periodic invariant tori. 

Namely, given a positive integer $d$, we consider the space $\mathbb{R}^d$ endowed with the standard sup-norm $||\cdot ||_\infty$, and we indicate by $B_R:=B_\infty(0,R) \subset \R^d$ the open ball of radius $R>0$ centered at the origin, and by $\mathbb{T}^d:=\mathbb{R}^d/\mathbb{Z}^d$ the $d$-dimensional torus. 
	Given $\ell\in \R$ and $D$ a domain of $\R^d$, we denote by $C^{\ell}(D)$  the space of bounded Hölder differentiable functions, endowed with the standard Hölder norm

 \begin{align}\label{norma holder}
	\begin{split}
	||f||_{C^{\ell}(D)}:=\sup_{\abs{\alpha} \le q}\, \sup_{x\in D}\ \abs{\partial^\alpha f(x)}+\sup_{\alpha\in\mathbb{N}^n| \,|\alpha|= q}\ \ \sup_{\substack{x,y \in D\\ 0<|x-y|<1}}\frac{|\partial^\alpha f(x)-\partial^\alpha f(y)|}{|x-y|^{\mu}}<+\infty\ ,
	\end{split}
	\end{align}
 where $q:=[\ell]$, $\mu:=\ell-q$, and where we have made use of standard multi-index notation.


Within this setting, for $\ell\ge 2$, we focus on Hamiltonians $H \in C^{\ell}(\mathbb{T}^d \times B_R)$ that - in standard action-angle coordinates - take  the form 
\begin{equation}
\label{eq_main_ham}
H(\theta, I)= \omega \cdot I + f(\theta, I)\ ,
\end{equation}
with $f(\theta, I)= \mathcal{O}(\theta,I^2)$. For any time $t\in \R$ for which it is defined, the associated flow starting at $(\theta,I)\in \T^d\times B_R$ is indicated by $\Phi^t_H(\theta,I)$.
It is known that, for Hamiltonians of the kind \eqref{eq_main_ham}, $\mathcal{T}_0= \mathbb{T}^d \times \{0\}$ is an invariant, Lagrangian quasi-periodic torus of frequency $\omega$ associated to $H$.  Moreover, we assume that $\omega$ is Diophantine, i.e. that there exist $\gamma>0$, $\tau\ge d-1$ such that
\begin{equation*}
 \abs{\omega \cdot k }\geq \frac{\gamma}{|k|^{\tau}}\; , \quad \forall \; k \in \Z^{d} \setminus \{0\}.
\end{equation*}
The set of vectors satisfying this condition for fixed values of $\tau$ and $\gamma$ is denoted by $\Omega^{d}_{\tau, \gamma}$, and we will also indicate $\Omega^{d}_{\tau}=\bigcup_{\gamma>0}\,\Omega^{d}_{\tau,\gamma}$. 


\medskip 
Within this setting, the main result of the present work is the following: 
\begin{theorem}
\label{thm_main2}
For any Hamiltonian $H\in C^\ell(\mathbb{T}^d\times B_R)$ as in \eqref{eq_main_ham}, with $\ell> 2d+1$  and $\omega \in \Omega^{d}_{\tau,\gamma}$, there exist constants $\rho^*, C_1 >0$  such that, for any $0<\rho<\rho^*$, and for any $(\theta_0, I_0)\in \T^d \times B_{\rho}$, the flow $\Phi_H^t$ verifies
\begin{equation}
\label{eq_estimate2}
\abs{\abs{\Pi_{I} \Phi^t_H(\theta_0,I_0)-I_0}}_\infty <\frac{\rho}{2} \quad \text{over a time } \quad 
|t|< t^{\text{stab}}= \frac{C_1}{\rho^{1+\frac{\ell-1}{\tau+1}} \abs{\log \rho}^{\ell-1}}.
\end{equation}
\end{theorem}

Expression \eqref{eq_estimate2} improves the estimates provided by Bounemoura in the finitely-differentiable case. Namely, in \cite[Corollary 3]{bou_eff_diff}, the author proved that - given an integer $k\ge 3$ - for any Hamiltonian of class $C^k$ around a $(\gamma,\tau)$-Diophantine torus - there exist $\rho^*>0$ and $C>0$ (depending only on $\gamma, \tau, k, d, R$) such that for all $0<\rho<\rho^*$,
\begin{equation}
 \label{eq_main_estimate}
  T(\rho):= \inf_{\theta_0\in \T^d, \abs{I_0}\leq \rho}\left\{ t>0,  \; \norm{\Pi_{I} \Phi^t_{H}(\theta_0, I_0)}_{\infty}=2\rho\right\} \geq  \frac{C}{\rho^{1+\frac{k-2}{\tau+1}}} .
\end{equation}

An immediate Corollary of Theorem \ref{thm_main2}, which follows by combining Theorem \ref{thm_main2} with Theorem A in \cite{Farre}, is the following.

\begin{corollary}
\label{thm_cor}
For any  real $\tau>d-1$,  $\ell>2d+1$  and for any pair of real constants $T_0>0$, $\varepsilon>0$, there exist
\begin{enumerate}
    \item a Hamiltonian $H\in C^\ell(\mathbb{T}^d \times B_R)$ of the form \eqref{eq_main_ham} with $\omega \in \Omega^{d}_{\tau}$\ ;
    \item a sequence of initial conditions $(\theta_n,I_n)$   verifying $\rho_n:=||I_n||\to 0 $;
    \item an associated sequence of times
    \begin{equation}\label{tempo_diffusione}
    t_n:=\displaystyle \frac{T_0}{\rho_n^{1+\frac{\ell-1}{\tau+1}+\varepsilon}}\ ;
    \end{equation}
    \item a pair of constants $C_1,c>0$;
\end{enumerate}
 such that, for every $n\ge 0$, the flow $\Phi^t_H(\theta, I)$ starting at any $(\theta,I)$ in a ball of radius $r:=\displaystyle\frac{|\rho_n|}{4}e^{-c|t_n|}$ around $(\theta_n,I_n)$ verifies
\begin{equation}\label{eq_Nekho_estimate}
    \norm{\Pi_I \Phi^t_H(\theta,I)}_\infty< \frac32 \rho_n\quad \text{for} \quad |t|<  \frac{C_1}{\rho_n^{1+\frac{\ell-1}{\tau+1}} \abs{\ln \rho_n}^{\ell-1}}
\end{equation}
and
\begin{equation}\label{eq_diff_estimate}
    \sup_{t\in [0,t_n]}\norm{\Pi_I \Phi^t_H(\theta, I)}_{\infty}\ge 2 \rho_n\ .
\end{equation}
\end{corollary}

\begin{remark}
The result from Theorem A in \cite{Farre} is stated for finitely differentiable Hamiltonian systems but it follows trivially from the proof that it does also hold for Hölder differentiable Hamiltonians.  This allows us to state the result from Corollary \ref{thm_cor} for any real $\ell>2d+1$ instead of only doing so for integer $\ell>2d+1$.
\end{remark}
Comparing \eqref{tempo_diffusione} with \eqref{eq_Nekho_estimate} and \eqref{eq_diff_estimate}, we see that Theorem \ref{thm_cor} provides almost optimal bounds of stability, in the sense that it ensures the existence of orbits whose stability and diffusion times match up to an arbitrarily small polynomial difference\footnote{The presence of the logarithmic factor in \eqref{eq_Nekho_estimate} is due to the use of analytic smoothing of Hölder functions in order to obtain stability estimates (see \cite{bamama}).}. As we anticipated above, the existence of small neighborhoods of points accumulating on a Diophantine torus and verifying estimate \eqref{eq_diff_estimate} was proved in Theorem A in \cite{Farre}, where examples of instability were built by making use of Anosov-Katok constructions (see \cite{AK, faka}). Therefore, up to an arbitrarily small polynomial correction, the present work does close the question of optimality for the effective stability estimates around quasi-periodic Diophantine tori of Hölder Hamiltonians.

In the literature, the question of finding upper bounds of stability and their optimality around Diophantine tori has been addressed also in 
the real-analytic and Gevrey classes. In these cases exponentially large upper bounds for the stability times have been found (see \cite{jorba, popov,  wiggins, po2}). The sharpness of the stability  exponents in these estimates has been proved in \cite{Fafa}. 

We stress that, in our case, no transversality conditions on the integrable part such as quasi-convexity or steepness\footnote{Steepness is a generic transversality condition on the gradient of sufficiently smooth functions. See \cite{bamama} for an introduction to its dynamical properties and \cite{Barbieri_2022,Barbieri_PhD} for a discussion and proof of its genericity.} are assumed. In particular, if quasi-convexity or steepness hold,  longer upper bounds of stability are known to exist (\cite{BFN, MG}), and the question of optimality for these cases requires different techniques such as Herman's synchronization method (see \cite{FS20}).


\medskip 

The stability estimates of Theorem \ref{thm_main2} are obtained by combining the improved analytic smoothing techniques introduced in \cite{bamama} together with standard normal form lemmas in analytic class (see \cite{po2}). The main idea (firstly introduced in \cite{bamama}) consists in regularizing Hamiltonian $H$ in \eqref{eq_main_ham} by making use of a suitable analytic smoothing Lemma for Hölder functions, and then in applying an analytic Normal Form Lemma (see \cite{po2}) to the smoothed Hamiltonian $H_s$. Clearly, this strategy works if the distance $H-H_s$ can be controlled in a suitable norm (see \cite[par. 4.3.2]{bamama} for more details on this technique). 
However, in the case under study, a direct implementation of the strategy in \cite{bamama} would not work and delicate modifications must carefully be considered in the construction.

In order to see heuristically what fails by  applying directly the techniques from \cite{bamama} to our case, we start by observing that, in order to smooth the function $H$ in \eqref{eq_main_ham},  one only needs to smooth the function $f$, whose expansion starts at order two in the actions. However, its smoothed counterpart, indicated by $f_s$, contains, in general, linear terms in the actions, and this hinders the whole construction. This difficulty is overcome here by truncating the Taylor expansion of $f$ (in the actions) around the origin, by smoothing the coefficients (which depend only on the angles) of the associated polynomial, and by suitably controlling the remainder of the initial truncation.
Moreover, it is worth noticing that - in the classical case of study - a rescaling of the action variables in order to pass from a small domain around the torus to a domain of order one is implemented. As we discuss in more detail in the next section, this would not work when analytic smoothing techniques are used. Moreover, working on a small domain in the action variables, in turn, requires extra care when estimating the size of the normal form transformation w.r.t. these coordinates. 

\medskip 

The rest of the paper is devoted to the proof of Theorem \ref{thm_main2}.

\section{Proof of Theorem \ref{thm_main2}}
We indicate by $\T^n_\C:=\C^n/\Z^n$ the complexification of the real torus. For any subset $U\subset \mathbb{T}^d\times \mathbb{R}^d$, we indicate its complex extension of analiticity widths $\sigma,\rho>0$ as 
\begin{equation}
    U_{\sigma, \rho}:= \left\{ (\theta, I) \in \mathbb{T}^d_\mathbb{C}\times \mathbb{C}^d \ |\ \; \sup_{I'\in \Pi_I U}|I - I'|_2 < \rho, \; ||Im(\theta)||_\infty< \sigma  \right\}\ ,
\end{equation}
where $|\cdot|_2$ indicates the standard euclidean norm. 

 Now, we introduce a couple of lemmas that will be used in the proof of the Theorem \ref{thm_main2}. 
 
 We start with the following result on the analytic smoothing of Hölder functions defined on the torus. It is a special case of a more general statement concerning functions defined in $\T^d\times B_R$, whose complete proof can be found in \cite{bamama}. The latter, in turn, is an improved version\footnote{Estimate \eqref{eq_fourier}, in particular, is a special case of a new highly non-trivial estimate \cite[formula (4.22)]{bamama}.} of analytic smoothing results due to Jackson, Moser and Zehnder (see \cite{Salamon,Chierchia}).

\begin{lemma}[Analytic smoothing]
\label{lemma_smoothing}
Fix an integer $d\geq 1$ and $s\in (0,1]$. Let $g\in C^\ell(\mathbb{T}^d )$, with $\ell> 2d+1$. Then there exist two constants $C_A=C_{A}(\ell,d)$ and $C_B=C_{B}(\ell,d)$ and an analytic function $\mathbf{g}_s$ on the closed complex extension $\overline{ \mathbb{T}^d_{s}}$ whose distance to the original function $g$ is controlled, for any $p\in \mathbb{N}$, $0\leq p \leq \ell$, by the estimate
\begin{equation}
\label{eq_estimate_smoothing}
\norm{g-\mathbf{g}_s}_{C^{p}(\mathbb{T}^d)}\leq C_A s^{\ell-p} \norm{g}_{C^\ell(\mathbb{T}^d)}, 
\end{equation}
 and whose Fourier norm verifies the non-trivial equality and bound
\begin{equation}\label{eq_fourier}
|||\mathbf{g}_s|||_{s}:=\sum_{\substack{k\in \Z^d } } |({\mathbf{\widehat{g}}_s)_k}| e^{|k|s}=\sum_{\substack{k\in \Z^d \\ \abs{k_1}+ \ldots + \abs{k_d}\leq 1/s} } |\widehat{g}_k| e^{|k|s}\leq C_{B} \norm{g}_{C^\ell(\mathbb{T}^d)}\ .
\end{equation}
\end{lemma}

\begin{proof}
We refer the reader to \textit{the periodic case} in \cite[pag. 365]{bamama}.
\end{proof}

\begin{remark} The function $\mathbf{g}_s$ is analytic on the boundary\footnote{It is actually entire in $\mathbb{T}^d_\mathbb{C}$, but outside of $\T^d_s$ its size grows too quickly to be useful, see \cite{bamama}} of $\T^d_s$, so that in the Fourier norm in \eqref{eq_fourier} we can choose an analyticity width exactly equal to $s$, instead of restricting to smaller domains as one usually does in these cases (see \cite{bamama}).  
\end{remark}

We will also need to use the classic Normal form Lemma for analytic functions, which was proved in \cite{po2}. Below we give a restatement of this result that can be found in \cite{bamama}.

Consider two numbers $\alpha,K>0$. We will say that $\omega\in \mathbb{R}^d$ is $(\alpha, K)$ completely non-resonant if, for all $k \in \Z^d\setminus\{0\}$ with $\sum_{i=1}^n |k_i|\leq K$, one has $\abs{k \cdot \omega}\geq \alpha$.  Now, take $\rho_0,\sigma, \sigma_0 >0$, with $\sigma_0>\sigma$, and let $\mathbf{H}(\theta,I) = \mathbf{h}(I) + \mathbf{f} (\theta,I)$ be an analytic Hamiltonian on the complex extension
$(\mathbb{T}^d\times D)_{\sigma_0, \rho_0}$,
where $D$ is an open set of $\mathbb{R}^d$ such that, for any  $I\in D$, $\omega(I):=\nabla \mathbf{h}(I)$ is $(\alpha, K)$ completely non-resonant. Also, let $M$ denote an upper bound for the
hermitian norm of the Hessian of $\mathbf{h}$ over $\Pi_I((\mathbb{T}^d\times D)_{\sigma_0, \rho_0})$. Then the following result holds.
\begin{lemma}[Normal form Lemma]
\label{lemma_pöschel}
If, for some $\rho > 0$ and $\xi>1$ one is ensured
\begin{equation}
\label{eq_pöschel_conditions}
|||\mathbf f|||_{\sigma, \rho}:= \sup_{I\in \Pi_I(\T^d\times D)_{\sigma,\rho}}\sum_{k\in \Z^d} |\widehat{\mathbf  f}_k(I)| e^{|k|\sigma}  \leq   \frac{1}{256\xi}\frac{\alpha \rho}{K}, \quad \rho \leq \min \left(\rho_0, \frac{\alpha}{2\xi M K}\right) , \quad K \sigma \geq 6
\end{equation}
then there exists a real-analytic, symplectic transformation $\Psi $ mapping the set $ (\mathbb{T}^d\times D )_{\sigma/6, \rho/2}$ into $ (\mathbb{T}^d\times D )_{\sigma, \rho}$, and taking $\mathbf H$ into resonant
normal form, that is
\begin{equation}
\label{eq_ineq_po}
 \mathbf H \circ \Psi(\phi, J) = \mathbf h(J)+\mathbf f^{*}(\phi, J), 
\quad 
|||\mathbf f^{*}|||_{\sigma/6, \rho/2} \leq e^{-K\sigma/6}|||\mathbf f|||_{\sigma, \rho}.
\end{equation}
Furthermore, $\Psi$ is close to the identity, in the sense that, for any $(\phi, J) \in (\mathbb{T}^d\times D )_{\sigma/6, \rho/2}$, one has
\begin{equation} 
\label{eq_diff_id}
\frac{\abs{\Pi_J \Psi - J}_{2}}{\rho} \leq 2^3 \frac{K}{\alpha \rho}|||\mathbf f|||_{\sigma, \rho}  \leq \frac{1}{32 \xi},
\quad \frac{\norm{\Pi_{\phi} \Psi - \phi}_{\infty}}{\sigma} \leq \frac{2^5K}{3\alpha \rho}|||\mathbf f|||_{\sigma, \rho}\leq \frac{1}{24\xi}\ .
\end{equation}

\end{lemma}

\medskip 
Lemma \ref{lemma_smoothing} and Lemma \ref{lemma_pöschel} will allow us to prove Theorem \ref{thm_main2}.
\begin{proofof}{Theorem \ref{thm_main2}}
First of all, let us rewrite the Hamiltonian given in \eqref{eq_main_ham}, 
\begin{equation*}
H(\theta, I)=\omega\cdot I+f(\theta, I)\ .
\end{equation*}
By expanding $f$ in Taylor series, we can split it into two terms as follows,
\begin{equation*}
f(\theta, I)=P_{\ell-2}(\theta, I)+Z(\theta, I),
\end{equation*}
where
\begin{align}\label{eq_exp_p}
&P_{\ell-2}(\theta, I):=\sum_{\substack{k\in \N^d,\\  2\leq \abs{k} \leq [\ell]-2}}{a_k(\theta) I^k} 
\quad , \qquad Z(\theta, I):=\sum_{\substack{\beta \in \N^d, \\ \abs{\beta}=[\ell]-1}}{R_{\beta}(\theta, I) I^{\beta}} 
\end{align}
and
\begin{equation*}
 a_{k}(\theta):=\frac{1}{k!} \partial_{I}^{\abs{k}}f(\theta,0)\quad, \qquad 
 R_{\beta}(\theta, I):= \frac{\abs{\beta}}{\beta !}\int_{0}^{1}{(1-t)^{\abs{\beta}-1} \text{D}^{\beta}f(\theta, tI) \; dt}.
\end{equation*}
Next, for a fixed $s\in (0,1)$ we apply Lemma \ref{lemma_smoothing} to the coefficients $a_k(\theta)$ in order to obtain real-analytic coefficients $a_{k,s} \in C^{\omega}(\T^d_s)$ and positive constants $C_A$ and $C_{B}$ such that\footnote{Differently from \cite{bamama}, a direct application of Lemma \ref{lemma_smoothing} to the function $f$ would yield a function containing linear terms w.r.t. the actions (see also the introduction) and thus hinder the whole construction. In order to see this from a heuristic point of view, one should take into account the fact that the analytic smoothing $F_s$ of a given Hölder function $F$ on $\R^d$ is given by its convolution with a kernel $K$ which is, in turn, the anti-Fourier transform of a bump function $\Phi: \R^d \longrightarrow \R$ (see \cite{Chierchia}). In formulas $F_s(x):=\int_{\R^d} K(x/s-y)F(y)dy$, with $K(x)=\int_{\R^d} e^{i\eta x}\Phi(\eta)d\eta$. It is clear by these formulas that, if the expansion of $F$ starts at order two, that of $F_s$ starts at order one in general.}
\begin{equation}
\label{eq_bound_four}
|||a_{k,s}|||_{s} \leq  C_{B}  \norm{a_{k}}_{C^{\ell}(\mathbb{T}^d). }
\end{equation}
and
\begin{equation}
\label{eq_bound_difference}
\norm{a_{k}-a_{k,s}}_{C^{1}(\T^d )}\leq C_A s^{\ell-1}\norm{a_{k}}_{C^\ell(\T^d)}.
\end{equation}
If we now define 
\begin{equation}
\label{eq_def_pl}
 (P_{\ell-2})_s(\theta, I):=\sum_{\substack{2\leq \abs{k} \leq [\ell]-2 \\ k\in \N^d} }{a_{k,s}(\theta) I^k}
\end{equation}
we can then rewrite $H$ as 
\begin{equation*}
H(\theta, I)= \omega \cdot I + (P_{\ell-2})_s(\theta, I)
+ ( P_{\ell-2}(\theta, I) - (P_{\ell-2})_s(\theta, I))+ Z(\theta, I).
\end{equation*}
Now, let us fix a real $\rho_0 \leq s$, and an integer $K\ge 1$. We consider the near-to-identity symplectic change of variables $\Psi$, obtained by applying Lemma \ref{lemma_pöschel} to the analytic Hamiltonian 
$$
\omega\cdot I+ (P_{\ell-2})_s \in C^{\omega}(\mathbb{A}^d_{s,\rho_0})\quad , \qquad \mathbb A:= \T^d\times \R^d
$$ in the smaller $(\alpha:= \gamma/K^\tau,K)$ non-resonant domain $\T^d\times D:= \T^d\times B_\rho$ around the Diophantine torus $\T^d\times \{0\}$.
In order for this to make sense, we need to fix $\rho\leq \rho_0$ sufficiently small so that the conditions \eqref{eq_pöschel_conditions} of Lemma \ref{lemma_pöschel}
\begin{equation}
\label{eq_p_con}
|||(P_{\ell-2})_s|||_{s, \rho}  \leq   \frac{1}{256\xi}\frac{\alpha \rho}{K}, \quad \rho \leq \min \left(\rho_0, \frac{\alpha}{2\xi M K}\right) , \quad K s \geq 6
\end{equation}
are satisfied.

It follows from \eqref{eq_def_pl}, \eqref{eq_bound_four} and a standard computation that there exists a positive constant $\mathtt C_0=\mathtt C_0(d)$ such that the Fourier norm of $(P_{\ell-2})_s$ verifies
\begin{equation}\label{eq_smooth_fourier}
||| (P_{\ell-2})_s |||_{s, \rho} \leq \mathtt C_0\, C_{B} \max_{\substack{2\leq \abs{k} \leq [\ell]-2 \\ k\in \N^d}}\norm{ a_{k,s}}_{C^\ell(\T^d)}\,\rho^2,
\end{equation}
leading to the condition, by imposing the first equation in \eqref{eq_p_con}, that
\begin{equation}
\label{eq_cond_1}
\rho \leq \frac{1 }{256\xi \mathtt{C}_0 C_{B} \max_{\substack{2\leq \abs{k} \leq [\ell]-2 \\ k\in \N^d}}\norm{ a_{k,s}}_{C^\ell(\T^d)}}\,\frac{\alpha}{K}.
\end{equation}
If we set, for $\tau\ge d-1$, for $a>0$, $b\geq 1$, and for a suitable constant $\tilde\rho>0$
\begin{equation}
\label{eq_our_choices}
\alpha=\frac{\gamma}{K^{\tau}}, \quad K=\left( \frac{\tilde{\rho}}{\rho}\right)^{a}, \quad
s=\left(\frac{\rho}{\tilde{\rho}} \right)^{a}\abs{\log \rho^b}
\end{equation}
 then \eqref{eq_cond_1} becomes
\begin{equation}
\label{eq_aaa}
\rho\leq \frac{1}{256\xi}\frac{\gamma }{\mathtt{C}_0 C_{B}    \max_{\substack{2\leq \abs{k} \leq [\ell]-2 \\ k\in \N^d}}\norm{ a_{k,s}}_{C^\ell(\T^d)}}\left(\frac{\rho}{\tilde \rho}\right)^{a(\tau+1)}.
\end{equation}
Thus if we choose 
\begin{equation}
a=\frac{1}{\tau+1}, \quad \tilde{\rho}=\left(\frac{\gamma }{256 \xi\, \mathtt{C}_0 C_{B}  \max_{\substack{2\leq \abs{k} \leq [\ell]-2 \\ k\in \N^d}} \norm{a_{k,s}}_{C^\ell(\T^d)}}\right)^{\frac{1}{a(\tau+1)}}
\end{equation}
then equation \eqref{eq_aaa} is satisfied. We also observe that - with the choices in \eqref{eq_our_choices}, by the fact that $b\ge 1$, and since $M>0$ can be taken arbitrarily close to zero as the integrable part of our hamiltonian is linear - the second and third inequalities in \eqref{eq_p_con} are satisfied for $\rho<\min\{\rho_0, e^{-6}\}$. 

Hence, by Lemma \ref{lemma_pöschel} 
 we obtain a symplectic change of variables 
 \[\Psi:(\mathbb{T}^d\times D )_{s/6, \rho/2} \to (\mathbb{T}^d\times D )_{s, \rho} \]
 such that 
\begin{align*}
&\tilde{H}(\phi, J):= H\circ \Psi (\phi, J)\\
&=\omega \cdot J+  
 (P_{\ell-2})_s \circ \Psi (\phi, J)
+( P_{\ell-2} - (P_{\ell-2})_s) \circ \Psi (\phi, J)+ Z \circ \Psi(\phi, J).
\end{align*}
From Pöschel's normal form lemma, the analytic part of the Hamiltonian is mapped into
\[ \omega \cdot J+(P_{\ell-2})_s \circ \Psi (\phi, J)= \mathbf{h}(J)+ \mathbf{f}^*(\phi, J)\] 
with $\mathbf{h}$ and $\mathbf{f}^*$ satisfying the inequalities in \eqref{eq_ineq_po}.
Therefore, we are reduced to study the dynamics of the action variables under the Hamiltonian
\begin{equation}
\label{eq_nouham}
 \tilde{H}(\phi, J)= \mathbf{h}(J) + \mathbf{f}^*(\phi, J)
+( P_{\ell-2} - (P_{\ell-2})_s) \circ \Psi (\phi, J)+ Z \circ \Psi(\phi, J).
\end{equation}
In particular, our next goal is to bound the partial derivative with respect to $\phi$ of the non-integrable part of \eqref{eq_nouham}, as it is this quantity which controls the drift of the actions variables.
We do this separately for the three terms above: the analytic remainder
$\mathbf{f}^*(\phi, J)$, the remainder $( P_{\ell-2} - (P_{\ell-2})_s) \circ \Psi(\phi, J)$ originating from the smoothing technique, and the remainder $Z \circ \Psi(\phi, J)$ coming from the initial truncation. 

For the first term, notice that, due to Lemma \ref{lemma_pöschel}, to estimate \eqref{eq_smooth_fourier}, and to the choice in \eqref{eq_our_choices}, one has the bound
\begin{align}\label{bound_resto_analitico}
\norm{\mathbf{f}^*(\phi, J)}_{s/6, \rho/2}&\leq e^{-Ks/6} ||| (P_{\ell-2})_s |||_{s, \rho} \leq   \mathtt{C}_0\,  C_{B} \max_{\substack{2\leq \abs{k} \leq [\ell]-2 \\ k\in \N^d}}\norm{ a_{k,s}}_{C^\ell(\T^d)}\, \rho^{2+\frac{b}{6}}
\end{align}
and thus by making use of the  Cauchy inequalities, for any $i=1,\ldots,d$,
\begin{equation*}
\norm{\partial_{\phi_i}\mathbf{f}^*(\phi, J)}_{s/12, \rho/2} \leq \mathtt C_1 \rho^{2+\frac{b}{6}-a}/|\log(\rho^{b})|\ ,
\end{equation*}
for some $\mathtt  C_1>0$ depending only on the initial data of the problem.

For the second term we firstly observe that, for any $i=1,\ldots, d$,
\begin{align}\label{catena}
\begin{split}
\partial_{\phi_i} ( P_{\ell-2} -(P_{\ell-2})_s)) \circ \Psi(\phi, J)=& \nabla_{J}( P_{\ell-2} -(P_{\ell-2})_s))(\Psi(\phi,J)) \partial_{\phi_i} \Pi_{J}\Psi \\
+&\nabla_{\phi}( P_{\ell-2} - (P_{\ell-2})_s))(\Psi(\phi,J)) \partial_{\phi_i} \Pi_{\phi}\Psi \\
=& \nabla_{J}( P_{\ell-2} - (P_{\ell-2})_s))(\Psi(\phi,J)) \partial_{\phi_i}\Pi_{J}\Psi \\
+& \nabla_{\phi}( P_{\ell-2} - (P_{\ell-2})_s))(\Psi(\phi,J)) (\partial_{\phi_i}\Pi_{\phi}\Psi-1) \\
+& \nabla_{\phi}( P_{\ell-2} - (P_{\ell-2})_s)(\Psi(\phi,J)).
\end{split}
\end{align}
Then, by using Cauchy inequalities and \eqref{eq_diff_id} we obtain that there exists a positive universal constant $\mathtt C_2>0$ such that
\begin{align}
\begin{split}
\norm{\partial_{\phi_i}\Pi_{J} \Psi(\phi, J)}_{s/12, \rho/2} \leq \mathtt C_2 \rho/s\quad , \qquad \norm{\partial_{\phi_i}\Pi_{\phi}\Psi-1}_{s/12, \rho/2} \leq \mathtt C_2.  \label{eq_c2}
\end{split}
\end{align}
Now, by recalling \eqref{eq_our_choices} and by using \eqref{eq_exp_p} and \eqref{eq_bound_difference}, we obtain that there exists a constant $\mathtt{C}_3>0$ depending only on the initial data of the problem such that
\begin{align}
\begin{split}
& \norm{ \nabla_{J}( P_{\ell-2} -(P_{\ell-2})_s))(\Psi(\phi,J)) }_{C^0(\T^d \times B_{\rho/2})} \leq \mathtt C_3 \rho^{1+a\ell}\abs{\log(\rho^b)}^{\ell},  \\
& \norm{\nabla_{\phi}( P_{\ell-2} -(P_{\ell-2})_s))(\Psi(\phi,J)) }_{C^0(\T^d \times B_{\rho/2})} \leq \mathtt C_3  \rho^{2+a(\ell-1)}\abs{\log(\rho^b)}^{\ell-1}. \label{eq_c4}
\end{split}
\end{align}
Thus, it follows by combining \eqref{catena} with the bounds in equations \eqref{eq_c2} and \eqref{eq_c4} that there exists a constant $\mathtt C_4>0$ such that \footnote{We observe that the first bound in \eqref{eq_c4} is larger than the second one by a factor $\rho$, due to the fact that on the l.h.s. one has a derivative w.r.t. the actions $J$ which make the corresponding function start at order one in $J$, while the second term to be estimated in \eqref{eq_c4} starts at order two in $J$. If the size of the normal form w.r.t. the action coordinates (first bound in \eqref{eq_c2}) were estimated in the standard way for this kind of computations (i.e. with a bound of order one, analogously  to the second estimate on the angle shift in \eqref{eq_c2}), one would have had a bound of order $O(\rho^{1+a(\ell-1)})$ in formula \eqref{cautela} which would have worsened the time estimates of the whole theorem. Instead, the correct estimation in \eqref{eq_c2} yields an identical bound of order $O(\rho^{2+a(\ell-1)})$ for all terms at the r.h.s. of \eqref{cautela}. This hidden (but important!) detail is due to the fact that - differently from the usual case - we could not rescale the action variables at the beginning in order to have a domain of order one. This, in turn, is due to the fact that - if $\rho$ were a quantity of order one - the first estimate in \eqref{eq_c2} would explode, as $s$ must be a small quantity in order for the analytic smoothing technique to be useful (see \eqref{eq_bound_difference}).},
for any $i=1,\ldots, d$, 
\begin{align}\label{cautela}
\begin{split}
&\norm{\partial_{\phi_i} ( P_{\ell-2} - \tilde{P}_{\ell-2,s}) \circ \Psi(\phi, J)}_{C^0(\T^d \times B_{\rho/2})}\\ 
&\leq
\norm{ \nabla_{J}( P_{\ell-2} -(P_{\ell-2})_s))(\Psi(\phi,J)) }_{C^0(\T^d \times B_{\rho/2})} \ \norm{\partial_{\phi_i}\Pi_{J} \Psi^*(\phi, J)}_{s/12, \rho/2}
\\
&+\norm{\nabla_{\phi}( P_{\ell-2} - (P_{\ell-2})_s))(\Psi(\phi,J))}_{C^0(\T^d \times B_{\rho/2} )} \ \norm{\partial_{\phi_i}\Pi_{\phi}\Psi-1}_{s/12, \rho/2} \\
&+ \norm{\nabla_{\phi}( P_{\ell-2} - (P_{\ell-2})_s)(\Psi(\phi,J))}_{C^0(\T^d \times B_{\rho/2} )} \leq \mathtt C_4 \rho^{2+a(\ell-1)}\abs{\log(\rho^b)}^{\ell-1}.
\end{split} 
\end{align}
Finally it follows easily from the second expression in \eqref{eq_exp_p} that, for any $i=1,\ldots,d$,
\begin{equation}\label{boundZ}
\norm{\partial_{\phi_i} Z\circ \Psi(\phi, J)}_{s/6, \rho/2} \leq \mathtt C_5\rho^{\ell-1},
\end{equation}
for some constant $\mathtt C_5>0$. Thus we arrive at the conclusion that, for any 
initial condition $(\phi_0, J_0)\in \T^d \times B_{7\rho/6}$ and for any time $t$ for which the flow is well-defined, the associated solution $\Phi_{\tilde{H}}^t(\phi_0, J_0)$ verifies 
\begin{align}
 &\norm{\Pi_{J}\Phi_{\tilde{H}}^t(\phi_0, J_0)- J_0}_\infty \nonumber \\
 &\leq \int_0^{t}\sup_{i\in\{1,\dots,d\}}{\abs{\partial_{\phi_i} [\mathbf{f}^*(\phi, J)
+( P_{\ell-2} - (P_{\ell-2})_s) \circ \Psi (\phi, J)+ Z \circ \Psi(\phi, J)]} \; ds} \nonumber \\
&\leq |t| ( \mathtt C_1 \rho^{2+\frac{b}{6}-a}/|\log(\rho^b)|+ \mathtt C_4 \rho^{2+a(\ell-1)}\abs{\log(\rho^b)}^{\ell-1}+ \mathtt C_5\rho^{\ell-1}). \label{eq_dominant}
\end{align}
It follows from a computation that, since we are assuming $d \geq 2$ and $\ell>2d+1$ \footnote{Originally, the condition $\ell>2d+1$ in Theorem \ref{thm_main2} is needed in order for estimate \eqref{eq_fourier} in Lemma \ref{lemma_smoothing} to hold (see \cite[par. 4.2]{bamama}).   }, then estimate
\[\ell-1>2+a(\ell-1) \quad \Longleftrightarrow \quad \ell >\frac{3-a}{1-a}=3+\frac{2}{\tau}\ge 3+\frac{2}{d-1}\ge 5\ \]\label{elle}is automatically verified, so that the term in \eqref{boundZ} is dominated by the one in \eqref{cautela} for sufficiently small $\rho$.
Moreover, if we choose $b=6(a\ell+1)$ it follows that also the bound in \eqref{bound_resto_analitico} is smaller than \eqref{cautela}, so that, finally there exists a constant $\mathtt C_6>0$ such that estimate \eqref{eq_dominant} can be rewritten as
\begin{equation*}
\norm{\Pi_{J}\Phi_{\tilde{H}}^t(\phi_0, J_0)- J_{0}}_\infty \leq |t| \mathtt C_6 \rho^{2+a(\ell-1)}\abs{\log(\rho^b)}^{\ell-1}.
\end{equation*}
This implies that for any time 
\begin{equation*}
t< t^*=\frac{1}{6 \,\mathtt C_6 \rho^{1+a(\ell-1)} \abs{\log(\rho^b)}^{\ell-1}}
\end{equation*}
the solution associated to any (real) initial condition $(\phi_0, J_0)\in \T^d \times B_{\rho+\rho/6}$ is well defined and satisfies
\begin{equation*}
\sup_{t\in[0, t^*]}\norm{\Pi_{J}\Phi_{\tilde{H}}^t(\phi_0, J_0)- J_0} < \rho/6.
\end{equation*}
It also follows from \eqref{eq_diff_id} that for $(\phi, J) \in (\mathbb{T}^d\times B_\rho )_{s/6, \rho/2}$ the size of the normal form in the actions is bounded by 
\[\norm{\Pi_J\Psi-J}_{\infty} \leq \frac{\rho}{6},\]
and so we obtain that 
in the original variables, for any initial condition $(\theta_0, I_0)\in \T^d\times B_{\rho}$, and for any time $t$ such that $|t|<t^*$ as above, one has
\begin{equation*}
\norm{I(t)-I_0}\leq \norm{I(t)-J(t)}+\norm{J(t)-J_0}+\norm{J_0-I(0)}<  3\times \frac{\rho}{6}=\frac{\rho}{2}\ ,
\end{equation*}
whence 
\begin{equation*}
\norm{\Pi_I\Phi_{H}^t(\theta_0, I_0)-I_0}< \frac{\rho}{2}\ . 
\end{equation*} 
This concludes the proof.
\end{proofof}

\noindent
\textbf{Acknowledgements:} The authors are grateful to L. Niederman for useful discussions. S. Barbieri has been supported by the ERC project 757802 Haminstab - Instabilities and homoclinic phenomena in Hamiltonian systems and by ICREA Premi Academia 2018 of prof. M. Guardia. G. Farré has been supported by the Juan de la Cierva-Formación Program (FJC2021-047377-I).

\bibliographystyle{plain}
\bibliography{Steep_Holder}

\medskip
\medskip

\noindent
Santiago Barbieri \\
Departament de Matem\`atiques i Inform\`atica\\
Universitat de Barcelona\\ Gran Via de les Corts Catalanes, 585\\
08007 Barcelona\\ Spain. \\
barbieri@ub.edu
\\
\\
Gerard Farré \\
Departament de Matemàtiques \\ Universitat Politècnica de Catalunya \\ Avinguda Diagonal, 647 \\ 08028 Barcelona \\ Spain.\\
gerard.farre.puiggali@upc.edu\\

\end{document}